\documentclass{amsart}
\usepackage[top=0.75in]{geometry}
\usepackage{amsmath,amsthm,amssymb, float}
\usepackage{enumerate, xcolor}
\usepackage{comment, hyperref}
\usepackage{todonotes}
\usepackage{tabularray}

\newtheorem{thm}{Theorem}[section]
\newtheorem{lem}[thm]{Lemma}

\newtheorem{cor}[thm]{Corollary}
\newtheorem{conj}{Conjecture}

\theoremstyle{definition}

\DeclareMathOperator{\mex}{mex}
\DeclareMathOperator{\opt}{opt}

\DeclareMathOperator{\dist}{dist}

\title{Two Games on Arithmetic Functions: \textsc{SALIQUANT} and \textsc{NONTOTIENT}}

\begin{document}

\begin{abstract}
    We investigate the Sprague-Grundy sequences for two normal-play impartial games based on arithmetic functions, first described by Iannucci and Larsson in \cite{sum}. In each game, the set of positions is $\mathbb{N}$. In \textsc{saliquant}, the options are to subtract a non-divisor.  Here we obtain several nice number theoretic lemmas, a fundamental theorem, and two conjectures about the eventual density of Sprague-Grundy values.
    
    In \textsc{nontotient}, the only option is to subtract the number of relatively prime residues.  Here are able to calculate certain Sprague-Grundy values, and start to understand an appropriate class function.   
\end{abstract}

\author[P. Ellis]{Paul Ellis}
\address{Department of Mathematics, Rutgers University,
110 Frelinghuysen Road,
Piscataway, NJ,
08854,
USA}
\email{paulellis@paulellis.org}
\author[J. Shi]{Jason Shi}
\address{}
\email{jshi3@caltech.edu}
\author[T. A. Thanatipanonda]{Thotsaporn Aek Thanatipanonda}
\address{Science Division, Mahidol University International College,
999 Phutthamonthon Sai 4 Rd, Salaya, Phutthamonthon District,
Nakhon Pathom,
73170,Thailand}
\email{thotsaporn@gmail.com}
\author[A. Tu]{Andrew Tu}
\address{}
\email{atu@brunswickschool.org}

\maketitle

\section{Introduction}

In this paper, we study two of the games introduced by \cite{sum}.  
Their rules are as follows.
\begin{enumerate}[(a)]
    \item \textsc{Saliquant}.  Subtract a non-divisor:  For $n\geq 1$, $\opt(n)=\{n-k : 1\leq k \leq n : k\nmid n\}$.
    \item \textsc{Nontotient}.  Subtract the number of relatively prime residues: For $n\geq 1$, $\opt(n)=\{n-\phi(n)\}$, where $\phi$ is Euler's totient function.
\end{enumerate}

In each case, we examine the normal-play variant only, so the usual Sprague-Grundy theory applies.  
In particular, the \emph{nim-value} of a position $n$ is recursively given by $$\mathcal{SG}(n)=\mex\{\mathcal{SG}(x)\mid x\in\opt(n)\},$$ where $\mex(A)$ is the least nonnegative integer not appearing in $A$.
Chapter 7 of \cite{LIP} gives a readable overview for the newcomer.  Note that for games of no choice, such as \textsc{nontotient}, $\mathcal{SG}(n)$ calculates the parity of the number of moves required to reach a terminal position.  The sole terminal position for \textsc{nontotient} is $1$.

\section{Let's play \textsc{saliquant}!}
%
%
%
%
Inaucci and Larsson give a uniform upper bound for nim-values of \textsc{saliquant} positions and show that odd positions attain this bound:

\begin{lem}[\cite{sum},Theorem 4]\label{lem-Sal-prev results}
In \textsc{saliquant},
\begin{itemize}
    \item If $n$ is odd, then $\mathcal{SG}(n)=\frac{n-1}{2}$
    \item For all $n\geq 1$, $\mathcal{SG}(n)<\frac{n}{2}$
\end{itemize}  
\end{lem}

Our task, therefore, will be to investigate the nim-values of even positions.  The first few such values are:

\noindent
\begin{tabular}{c|ccccccccccccccccccccc}
    $n$ & $2$ &$4$  & $6$&  $8$& $10$&  $12$& $14$& $16$& $18$&  $20$ & $22$ & $24$ &  $26$& $28$&  $30$&  $32$&  $34$& $36$&  $38$&  $40$ &  $42$ \\
    \hline
    $\mathcal{SG}(n)$ &  $0$  & $1$ & $1$ & $3$ &$2$ & $4$ & $6$ & $7$ & $4$ & $7$ & $5$ &  $10$ &  $12$ &
    $10$ &  $13$ & $15$ & $8$ & $13$ &  $9$ &  $17$ & $17$\\
\end{tabular}

First we can establish some particular cases where the nim-value will be below the uniform upper bound given in the last part of the Lemma:

\begin{lem}\quad

\begin{itemize}
    \item If $3 \mid n$, then $\mathcal{SG}(2n) \leq n - 2$. 
    \item If $5 \mid n$, then $\mathcal{SG}(4n) \leq 2n - 3$.
\end{itemize}
\end{lem}
\begin{proof}
If $3 \mid n$, then $2n - 4$ is the largest possible option of $2n$.  So by Lemma \ref{lem-Sal-prev results} all options have a nim-value of at most $n-3$.  Hence $\mathcal{SG}(2n) \leq n - 2$.

Similarly, if $5 \mid n$, then $4n - 6$ is the largest possible option of $4n$, with the exception of $4n-3$.  So all options have a nim-value of at most $2n-4$, or exactly $2n-2$.  Hence $\mathcal{SG}(4n) \leq 2n - 3$.
\end{proof}
Note that the former bound is sharp. For example, setting $n = 15$, we see that $\mathcal{SG}(30)$ is
$13$.
Next we establish a uniform lower bound.

\begin{lem}\label{lem-Sal-recursive-lower-bound}
If $p$ is the smallest prime divisor of $n$, then $\mathcal{SG}(n) \geq \mathcal{SG}(\frac {p-1}{p}n)$; in particular, $\mathcal{SG}(2n) \geq \mathcal{SG}(n)$.
\end{lem}

\begin{proof}
Let $n-k<\frac{p-1}{p}n$, where $p$ is the smallest prime divisor of $n$.  Then $\frac{n}{p}<k<n$, and so $k\nmid n$.  Hence $n-k$ is an option of $n$. $n$ has every option which $\frac{p-1}{p}n$ has. Thus, $\mathcal{SG}(n) \geq \mathcal{SG}(\frac{p-1}{p})n$.
\end{proof}

\begin{cor}\label{lem-Sal-first lb for even}
$\mathcal{SG}(n) \geq \frac{n-2}{4}$ for all $n$.
\end{cor}

\begin{proof}
Lemma \ref{lem-Sal-prev results} establishes this for odd $n$.  

Next if $n = 2k$, where $k$ is odd, then Lemma \ref{lem-Sal-recursive-lower-bound} tells us that $\mathcal{SG}(n) \geq \mathcal{SG}(k) = \frac{k-1}{2} = \frac{n-2}{4}$. 

Now let $n = 2^mk$, where $k$ is odd and $m \geq 2$.  Then $n-(k+2), n-(k+4), \ldots ,1$ are all options of $n$, with nim-values $\frac{1}{2}(n-k-3), \frac{1}{2}(n-k-5), \ldots, 0$, respectively.  Thus $$\mathcal{SG}(n)
\geq\frac{1}{2}\left(n-k-1\right)\geq\frac{1}{2}\left(n-\frac{n}{4}-1\right)
>\frac{1}{2}\left(\frac{n}{2}-1\right)>\frac{n-2}{4}$$
\end{proof}
Next we prove a key lemma about the nim-values of even positions.

\begin{lem}\label{lemma-sal-divisibility}
If $\mathcal{SG}(2n) = n - k$, then $2k - 1 \mid n$.
\end{lem}
\begin{proof}
Suppose $\mathcal{SG}(2n) = n - k$. Then $2n$ has no option of nim-value $n - k$. Since $\mathcal{SG}(2n - 2k + 1)=n-k$, it is not an option.  In other words, $2k - 1 \mid 2n$ and hence $2k - 1 \mid n$.
\end{proof}

From here, we can establish nim-values of several particular cases of even numbers.  
To start, the previous result immediately narrows the possibilities of double a prime or semiprime.

\begin{cor}\label{cor-sal-double a prime}
Let $p,q$ be odd primes, then 
\begin{itemize}
    \item $\mathcal{SG}(2p) = p - 1$ or $p - \frac{p+1}{2} = \frac{p-1}{2}$; and
    \item $\mathcal{SG}(2pq) = pq - 1, pq - \frac{p+1}{2}, pq - \frac{q+1}{2}$, or $pq - \frac{pq+1}{2} = \frac{pq-1}{2}$.
\end{itemize}
\end{cor}
We next refine the first bullet point.
\begin{lem}
Let $p\geq 5$ be prime.  Then the only possible option of $2p$ with nim-value $\frac{p-1}{2}$ is $p+1$.  Hence \begin{itemize}
    \item $\mathcal{SG}(p+1)=\frac{p-1}{2} \implies \mathcal{SG}(2p)=p-1$
    \item $\mathcal{SG}(p+1)\neq\frac{p-1}{2} \implies \mathcal{SG}(2p)=\frac{p-1}{2}.$
\end{itemize}
\end{lem}
Note that if $p=3$, then $p+1=4$ is not an option of $2p=6$.  
\begin{proof}
Let $x\in\opt(2p)$ such that $\mathcal{SG}(x) = \frac{p-1}{2}$.

By Lemma \ref{lem-Sal-prev results}, if $x$ were odd, then we would have $x=p$, but $p$ is not an option of $2p$.
So let $x=2n$ for some $n$.  By Lemma \ref{lemma-sal-divisibility}, since $\mathcal{SG}(2n)=\frac{p-1}{2}=n-(n-\frac{p-1}{2})$, we have $2(n-\frac{p-1}{2})-1 \mid n$, and so $2n-p \mid n$.  On one hand, this implies that $n<p$.

On the other hand, it means that we can write $n=d(2n-p)=2dn-dp$ for some $d\in\mathbb{N}$.  Then $n\mid dp$ and $d\mid n$.  Now since $n<p$ and $p$ is prime, we have $n\mid d$.  Finally, since $d\mid n$, this means $n=d$, and so $2n-p=1$ or $x=2n=p+1$.
\end{proof}

In fact, this is enough to generate infinitely many examples for which our uniform lower bound is attained.

\begin{thm}
 If $p$ is prime and $p \equiv 5 \bmod 6$, then $\mathcal{SG}(2p) = \frac{p-1}{2}$.
\end{thm}
\begin{proof}
Let $p$ be prime where $p \equiv 5 \bmod 6$.  We claim that $\mathcal{SG}(p+1)\neq\frac{p-1}{2}$, and so the previous lemma implies that $\mathcal{SG}(2p)=\frac{p-1}{2}$.

Indeed, since $p \equiv 5\bmod 6$, we have $p+1\equiv 0\bmod 6$.  In particular, $1,2,3\mid p+1$, so the largest possible option of $p+1$ is $(p+1)-4=p-3$.  So by Lemma \ref{lem-Sal-prev results}, for all $y\in \opt(p+1)$,  $\mathcal{SG}(y)<\frac{y}{2}\leq \frac{p-3}{2}$, that is $\mathcal{SG}(y)\leq \frac{p-5}{2}$.  Hence $\mathcal{SG}(p+1)\leq \frac{p-3}{2}$.
\end{proof}

\begin{cor}
There are infinitely many $n \in \mathbb{N}$ such that $\mathcal{SG}(n) = \frac{n-2}{4}$.
\end{cor}
\begin{proof}
It is well known that there are infinitely many primes $p = 5\bmod  6$. For each of these $p$, letting $n=2p$, we have $\mathcal{SG}(n) = \mathcal{SG}(2p) = \frac{p-1}{2}=\frac{n-2}{4}$.
\end{proof}

It is possible to keep refining this inquiry about numbers which are twice an odd.  For example Corollary \ref{cor-sal-double a prime} could be extended for more than $2$ odd prime factors, but we don't see how helpful it is.  Instead, we investigate the remaining cases by decomposing even numbers as an odd number times a power of $2$.
As a first step, we can compute exact nim-values in the case that the odd part is $1$, $3$, $5$, or $9$.

\begin{lem}\label{lem-smallcases}
Let $b\geq 1$. Then $\mathcal{SG}((2a+1)2^b) = (2a+1)2^{b-1} - a-1$ for $a = 0,1,2,4$.
\end{lem}

\begin{proof}
This can be checked by hand for the cases when $b=1$ or $b=2$, so let $b\geq 3$, and consider the options of $(2a+1)2^b$

All odd numbers greater than $(2a+1)$ are non-divisors of $(2a+1)2^b$, so the odd numbers $1,3,\ldots, (2a+1)2^b - (2a+3)$ are all options with nim-values $0,1,\ldots, (2a+1)2^{b-1}-a-2$, respectively. 

We claim that there is no option with nim-value $(2a+1)2^{b-1}-a-1$. Indeed $(2a+1)2^b - 2a - 1$ is not an option, and is the only odd number with nim-value $(2a+1)2^{b-1}-a-1$. Next, note that $b \geq 3$ and $a = 0,1,2,4$ i.e. $(2a+1) = 1, 3, 5, 9$. Hence all even numbers less than $(2a+1)$ divide $(2a+1)2^b$ and the only even options are less than or equal to $(2a+1)2^b-2a-2$. By Lemma \ref{lem-Sal-prev results}, their nim-values are less than $\frac{(2a+1)2^b-2a-2}{2} = (2a+1)2^{b-1}-a-1$. Thus, there is no option with nim-value $(2a-1)2^{b-1}-a-1$, and $\mathcal{SG}((2a-1)2^b) = (2a-1)2^{b-1}-a-1$. 
\end{proof}

We now see that there are infinitely many \emph{even} values for which our uniform upper bound is obtained:

\begin{cor}\label{sal-cor-powersof2}
Let $b\geq 1$. Then $\mathcal{SG}(2^b) = 2^{b-1} - 1$.  In particular, there are infinitely many $m$ for which $\mathcal{SG}(n)=\frac{n-2}{2}$.
\end{cor}

Note that the above proof does not work, for example, when $a = 3$ i.e. $ 2a+1=7$, since $6 < 2a + 1$, and $6 \nmid 7(2^b)$. In fact $\mathcal{SG}(14) = 6$, not $(2a+1)2^{b-1} -a-1=3$.  Next we obtain a slightly weaker result when $a=10$ and $ 2a+1=21$.

\begin{lem}\label{sal-21case}
Let $b\geq 1$. Then $\mathcal{SG}(21  (2^b)) = 21  (2^{b-1}) - 11$ or $21 (2^{b-1}) - 4$.
\end{lem}
\begin{proof}
In the case $b = 1$, we see $\mathcal{SG}(42) = 17$. For $b \geq 2$, consider the options of $21  (2^b)$.  The odd numbers $1,3,\ldots, 21(2^b) - 23$ and $ 21(2^b) - 19, \ldots, 21(2^b) - 9$ are all options with nim-values $0,1,\ldots, 21(2^{b-1})-12$ and $ 21(2^{b-1})-10, \ldots, 21(2^{b-1})-5$, respectively. The numbers  $21(2^b) - 21$ and $21(2^b) - 7$ with nim-values $21(2^{b-1}) - 11$ and $21(2^{b-1}) - 4$ are not options, and all larger odd numbers have nim-values greater than $21(2^{b-1}) - 4$.

On the other hand, since $2,4,6\mid 21 (2^b)$, Lemma \ref{lem-Sal-prev results} implies that any even options have nim-values less than $\frac{21(2^b)-8}{2} = 21(2^{b-1}) - 4$.  Hence $\mathcal{SG}(21  (2^b)) = 21 (2^{b-1}) - 11$ or $21 (2^{b-1}) - 4$.
\end{proof}

We end this section by showing that twice a Mersenne number is above the uniform lower bound.  Note that if $m=2n=2(2^b-1)$ then $\frac{m-2}{4}=\frac{n-1}{2}=2^{b-1}-1$.

\begin{lem}
Let $b\geq 3$.  Then $\mathcal{SG}(2(2^b-1))>2^{b-1}-1$.  In particular, if $2^b-1$ is prime, then $\mathcal{SG}(2(2^b-1))=2^b-2$.
\end{lem}

\begin{proof}
By Corollary \ref{sal-cor-powersof2}, $\mathcal{SG}(2^b)=2^{b-1}-1$, so we just need to show that $2^b\in\opt(2(2^b-1))$.

Suppose otherwise and that $2(2^b-1)-2^b\mid 2(2^b-1)$.  Then $2^{b}-2\mid 2(2^b-1)$.  Thus either $2^{b}-2$ and $2^b-1$ share a common factor and so $2^{b}-2=1$, or $2^{b}-2\mid 2$ and so $b\leq 2$.  Both cases are impossible.

In the case $2^b-1$ is prime, Corollary  \ref{cor-sal-double a prime} implies $\mathcal{SG}(2(2^b-1))=2^b-2$.
\end{proof}

\section{The Fundamental Theorem of \textsc{saliquant} and density of values}

Finally, we obtain our most general statement about nim-values of Saliquant.  The two corollaries which follow were actually proved first, inspired by the proof of Corollary \ref{lem-Sal-first lb for even}.  
\begin{thm}\label{sal-thm-big-one} For all $a\geq 0,b \geq 1$, 
\begin{align*}
\mathcal{SG}\left((2a+1)2^b\right)
&=\frac{m}{2m+1}\left((2a+1)2^b-1\right)+\frac{1}{2m+1}\left((2a+1)2^{b-1}-a-1\right)\\
&=(2a+1)2^{b-1}-\frac{1}{2}\left(\frac{2a+1}{2m+1}+1\right)
\end{align*}
for some non negative integer $m$. Thus
\begin{align*}
\mathcal{SG}\left((2a+1)2^b\right)&=(2a+1)2^{b-1}-\frac{d+1}{2}, \text{ where $d$ is a factor of $2a+1$}. 
\end{align*}

\end{thm}

This theorem unifies several edge cases, as well.  If we set $a=0$, then we must have $d=1$, obtaining Corollary \ref{sal-cor-powersof2}.  Let $f(a,b,m)$ be the function given by Theorem \ref{sal-thm-big-one}. If we set $b=0$, then $f(a,b,m)$ is never an integer, but $\lim_{m\to\infty} f(a,b,m)=\frac{n-1}{2}$, matching Lemma \ref{lem-Sal-prev results}.

Fixing $a$ and $b$, $f(a,b,m)$ is a linear rational function in $m$, thus monotonic for $m\geq 0$, and it is easily checked that it is increasing.  Hence its minimum is obtained when $m=0$, with an upper bound given by $m\to\infty$.
Thus we have the following corollary, which itself is a generalization of Lemma \ref{lem-smallcases}. 

\begin{cor}\label{sal-cor-big1}
For all $a,b \geq 1$, 
$$\frac{(2a+1)2^b}{2}-a-1\leq \mathcal{SG}\left((2a+1)2^b\right)< \frac{(2a+1)2^b}{2}-\frac{1}{2}.$$
\end{cor}

The upper bound is the same as in Lemma \ref{lem-Sal-prev results}.  If we fix $a$ and let $b$ grow large, the lower bound is an asymptotic improvement over Corollary \ref{lem-Sal-first lb for even} from $\mathcal{O}(\frac{n}{4})$ to $\mathcal{O}(\frac{n}{2})$.  Furthermore, we will see experimentally below that all values of $f(a,b,m)$ are obtained. To illustrate the theorem, set $b=1$ to obtain all possible nim-values of even numbers which are not multiples of $4$:

\begin{cor}\label{sal-cor-4a+2}
For all $a\geq 1$, $\mathcal{SG}(4a+2)$ must have the form $$\frac{(4m+1)a+m}{2m+1} \quad\left(=a, \frac{5a+1}{3}, \frac{9a+2}{5}, \frac{13a+3}{7},\frac{17a+4}{9}, \frac{21a+5}{11},\ldots\right)$$
for some $m\geq 0$.
\end{cor}

\begin{proof}[Proof of Theorem \ref{sal-thm-big-one}]

Suppose $a,b \geq 1$. Let $X=\mathcal{SG}\left((2a+1)2^b\right)$. Then $$X = ((2a+1)2^{b-1}) - ((2a+1)2^{b-1} - X),$$ so by Lemma \ref{lemma-sal-divisibility}, we have 
$$\left(2\left((2a+1)2^{b-1}-X\right) - 1\right) \mid (2a+1)2^{b-1}.$$                                                                                                                                                                                                                                      
Thus there is some $Q_1$ so that 
$$Q_1(a2^{b+1}+2^b-2X-1) = (2a+1)2^{b-1}.$$
Since $(a2^{b+1}+2^b-2X-1)$ is odd, $2^{b-1}\mid Q_1$.  Pick $Q_2$ so that $Q_2 2^{b-1}=Q_1$.  This gives
$$Q_2 (a2^{b+1}+2^b-2X-1)= 2a+1.$$
Next since $Q_2$ is odd, we can set $Q_2=2m+1$ for some $m\geq 0$, giving
$$(2m+1) (a2^{b+1}+2^b-2X-1)= 2a+1.$$
Finally, solving for $X$ gives the desired result.
\end{proof}

Now that we know the specific possible values $\mathcal{SG}(n)$ can take based on the decomposition $n=(2a+1)2^b$, a natural question is how these values are distributed. 
For a given $b>0$, $m\geq 0$, define $$S_{b,m}=\{a\in\mathbb{N}\mid \mathcal{SG}((2a+1)2^b)=f(a,b,m)\}.$$
The experimental density of $S_{b,m}$ for $b=1, 2, 3, 4$ and $m=0, 1, 2, 3, 4$ are shown in Table \ref{table-sal}. For $b=1$, we measured up to $a=5000$; for $b=2,3$, up to $a=2000$; and for $b=4$, up to $a=1000$.  The associated Maple program can be found at the third author’s website \url{http://www.thotsaporn.com}.

In Figure \ref{fig:salvalues}, we can see some of these values, with the corresponding labels given in Table \ref{table-sal}.  For example, consider the entry of the table marked \textbf{(C)}.  It says that the density of numbers of the form $x=8a+4$ for which $\mathcal{SG}(x)=3a+1$ is $0.561$.  Then we can see that the line in the figure with slope $\frac{3}{8}$ (also marked \textbf{(C)}) has about half density.  Contrast with the entry marked \textbf{(D)}, corresponding to the line with slope $\frac{5}{12}$.  It is very sparse, as seen in the figure.  Notice that the $y$-intercept of each of these lines corresponds to $a=-\frac{1}{2}$, which in each case gives
$$f\left(-\frac{1}{2},b,m\right)=\frac{1}{2m+1}\left(-m2^b-2^{b-1}+\frac{1}{2}+m2^b-m+2^{b-1}-1\right)=\frac{-m-\frac{1}{2}}{2m+1}=-\frac{1}{2},$$
which is ok to be negative, since the game is only meaningfully defined on positive numbers.
Finally, the line marked \textbf{(A)} is $y=\frac{x-1}{2}$, which includes all odd $x$ and some even $x$, per Lemma \ref{lem-Sal-prev results} and Corollary \ref{sal-cor-powersof2}.

\begin{table}[h!]
  \caption{Experimental values of $S_{b,m}$.  The labels \textbf{(B)---(E)} match Figure \ref{fig:salvalues}.}
  \begin{center}
    \begin{tabular}{c|ll|ll|ll|ll} 
       &$\mathcal{SG}(4a+2)$&density&$\mathcal{SG}(8a+4)$&density&$\mathcal{SG}(16a+8)$&density&$\mathcal{SG}(32a+16)$&density\\
       $m$&$=f(a,1,m)$ & $(a\leq 5000)$ & $=f(a,2,m)$ & $(a\leq 2000)$ & $=f(a,3,m)$ & $(a\leq 2000)$ & $=f(a,4,m)$ & $(a\leq 1000)$\\     
      \hline
      $0$&$a$ \;\textbf{(B)} & $0.532$ & $3a+1$\;\textbf{(C)} & $0.561$ & $7a+3$\;\textbf{(E)} & $0.540$ & $5a+7$ & $0.638$\\
      &&&&&&&&\\
      $1$&$\displaystyle\frac{5a+1}{3}$\;\textbf{(D)} & $0.026$ & $\displaystyle\frac{11a+4}{3}$ & $0.056$ & $\displaystyle\frac{23a+10}{3}$ & $0.090$ & $\displaystyle\frac{47a+22}{3}$ & $0.069$\\
      &&&&&&&&\\
      $2$&$\displaystyle\frac{9a+2}{5}$ & $0.037$ & $\displaystyle\frac{19a+7}{5}$ & $0.044$ & $\displaystyle\frac{39a+17}{5}$ & $0.050$ & $\displaystyle\frac{79a+37}{5}$ & $0.046$\\
      &&&&&&&&\\
      $3$&$\displaystyle\frac{13a+3}{7}$ & $0.061$ & $\displaystyle\frac{27a+10}{7}$ & $0.049$ & $\displaystyle\frac{55a+24}{7}$ & $0.046$ & $\displaystyle\frac{111a+52}{7}$ & $0.043$\\
      &&&&&&&&\\
      $4$&$\displaystyle\frac{17a+4}{9}$ & $0.022$ & $\displaystyle\frac{35a+13}{9}$ & $0.030$ & $\displaystyle\frac{71a+31}{9}$ & $0.015$ & $\displaystyle\frac{143a+67}{9}$ & $0.010$\\
    \end{tabular}\label{table-sal}
  \end{center}
\end{table}

\begin{figure}[h!]
    \centering
    \includegraphics[scale=0.6]{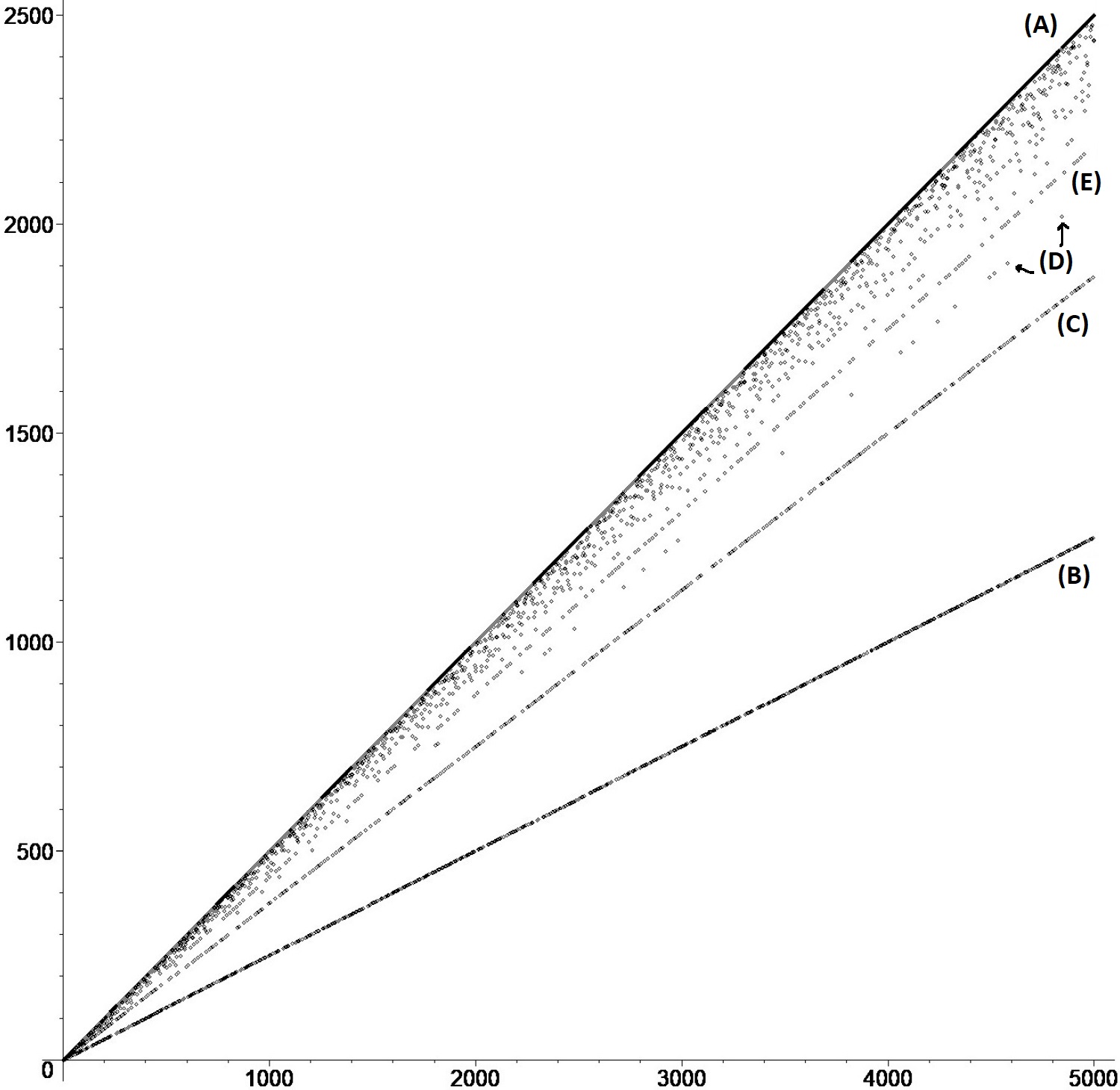}
    \caption{The first $5000$ nim-values of \textsc{saliquant}.  The slopes of the labelled lines are 
    \textbf{(A)} $\frac{1}{2}$, 
    \textbf{(B)} $\frac{1}{4}$, 
    \textbf{(C)} $\frac{3}{8}$, 
    \textbf{(D)} $\frac{5}{12}$, 
    \textbf{(E)} $\frac{7}{16}$
    }
    \label{fig:salvalues}
\end{figure}
We next show a straightforward upper bound for these densities, noting that each of the values in Table \ref{table-sal} are well below this bound.

\begin{lem}
Given $b\geq 1$, $m\geq 0$, the density of $S_{b,m}$ is at most $\frac{1}{2m+1}$.
\end{lem}
\begin{proof}
Fix $b\geq 1$, $m\geq 0$, and consider 
\begin{align*}
(2m+1)f(a,b,m)&=\left(a(m2^{b+1}+2^b-1) + m(2^b-1)+(2^{b-1}-1)\right)\\
&=a\left((2m+1)2^{b}-1\right)+(2m+1)2^{b-1}-m-1\\
&\equiv -a-m-1\bmod (2m+1)
\end{align*}
Thus $f(a,b,m)$ is only an integer when $a\equiv m\bmod (2m+1)$, and so $\frac{1}{2m+1}$ is an upper bound for how frequently $\mathcal{SG}((2a+1)2^b)$ can attain this value.
\end{proof}

\newpage

Given the values in Table \ref{table-sal}, we suspect that most of these values are actually $0$:

\begin{conj}
For a given $b>0$, 
\begin{itemize}
    \item If $m=0$, $S_{b,m}$ has positive density, and 
    \item If $m>0$, $S_{b,m}$ has density $0$, but is nonempty.
\end{itemize}
\end{conj}

We can also look at how fixing $a$ and $b$ affects the value of $m$.  Define $M(a,b) = m$ where $\mathcal{SG}((2a+1)2^b) = f(a,b,m)$, and consider Table \ref{table-sal-2}. 

\begin{table}[h!]
  \caption{Experimental values of $M(a,b)$.}
  \begin{center}
    \begin{tabular}{c|l|l|l|l|l|l|l|l|l} 
        &$a=3$&$a=4$&$a=5$&$a=6$&$a=7$&$a=8$&$a=9$&$a=10$&$a=11$ \\ 
      \hline
      $b=1$&$3$&$0$&$0$&$6$&$2$&$0$&$0$&$1$&$0$\\
      $b=2$&$0$&$0$&$0$&$0$&$0$&$8$&$0$&$1$&$0$\\
      $b=3$&$0$&$0$&$0$&$6$&$0$&$0$&$9$&$0$&$0$\\
      $b=4$&$0$&$0$&$0$&$0$&$0$&$0$&$0$&$0$&$11$\\
      $b=5$&$0$&$0$&$0$&$0$&$0$&$0$&$0$&$0$&$0$\\
      $b=6$&$0$&$0$&$0$&$0$&$0$&$0$&$9$&$0$&$0$\\
      $b=7$&$0$&$0$&$0$&$0$&$0$&$0$&$0$&$0$&$0$\\
      $b=8$&$0$&$0$&$0$&$0$&$0$&$0$&$0$&$0$&$0$\\
      $b=9$&$0$&$0$&$0$&$0$&$0$&$0$&$0$&$0$&$0$\\
      $b=10$&$0$&$0$&$0$&$0$&$0$&$0$&$0$&$0$&$0$\\
      
    \end{tabular}\label{table-sal-2}
  \end{center}
\end{table}
If one imagines running along any row of Table \ref{table-sal-2} and tracking the distribution of $m$, they would get the densities achieved in Table \ref{table-sal} as $a$ tends toward infinity. For our second conjecture, we instead consider the behavior of $m$ rather than of $a$, noting that it is difficult to generate more data as the values grow exponentially as $b$ increases.

Given the sparsity of each column, we suspect that in each column all but a finite number of values are non-zero.
\begin{conj}
    For a given $a>0$, for sufficiently large $b$, $M(a,b)=0$, in which case 
    
    $\mathcal{SG}((2a+1)2^b)=(2a+1)2^{b-1}-a-1$.
\end{conj}
Note that Lemma \ref{lem-smallcases} proves a stronger form of the conjecture for $a=0,1,2,4$, and Lemma \ref{sal-21case} shows that when $a=10$ we have either $m=0$ or $m=1$.

\section{NONTOTIENT}

Denoting $\phi(n)=\lvert\{1\leq k\leq n\mid k\text{ is not a factor of }n\}\rvert$, \cite{sum} also define two games based on $\phi(n)$:
\begin{itemize}
    \item \textsc{Totient}: $\opt(n)=\phi(n)$
    \item \textsc{Nontotient}: $\opt(n)=n-\phi(n)$
\end{itemize}

In this section we make some headway in understanding \textsc{nontotient}.  
First recall that $\phi(ab)=\phi(a)\phi(b)$, and for prime $p$, $\phi(p^k)=p^{k-1}(p-1)$.  Thus if $n=p^{k_1}_1\ldots p^{k_m}_m$, we have $\phi(n)=\prod {p_i}^{k_i - 1}(p_i -1)$.  In particular $\phi(1)=1$. Define $g(n):=\opt(n)=n-\phi(n)$.  We immediately obtain:

\begin{lem}\label{lem-nont-parity}
For $n>2$, $\phi(n)$ is even, and so $g(n)$ has the same parity as $n$.    
\end{lem}


For the rest of the section, let $p$ and $q$ always represent primes.  As noted in \cite{sum}, $g(p^k)=p^{k-1}$.  Hence the game on $p^k$ terminates after $k$ moves and so $\mathcal{SG}(p^k)=0$ if and only if $k$ is even.  They also note that $g(p^k q)=p^{k-1}(q+p-1)$, and so in the case that $q+p-1$ is a power of $p$, this becomes easy to compute.  Consider for example the prime pairs $(p,q)=(2,7)\text{ or }(3,7)$.  
We can extend this as follows.  First note that $$g(p^k q^l)=p^k q^l - p^{k-1}(p-1)q^{l-1}(q-1)=p^{k-1}q^{l-1}(p+q-1).$$  Then we have

\begin{thm}\label{lem-nont}\quad

\begin{enumerate}[(a)]
    \item If $q=p^b-p+1$ where $b$ is even, then $\mathcal{SG}(p^k q^l)=0$ if and only if $q$ is even.  
    \item If $q=p^b-p+1$ where $b$ is odd, then $\mathcal{SG}(p^k q^l)=0$ if and only if $q+l$ is even.
\end{enumerate}
\end{thm}

\begin{proof}
In this case $g(p^k q^l)=p^{k-1}q^{l-1}\left(p+(p^b-p+1)-1\right) = p^{k+b-1}q^{l-1}$. So after $l$ moves, the position will be $p^{k+l(b-1)}$, and thus the game terminates after $k+l(b-1)+l = k+lb$ moves. 
\end{proof}

Some prime pairs $(p,q)$ that satisfy part (a) are $(2,3)$, $(3,7)$, $(7,43)$, $(13, 157)$, $(3,79)$, $(11,14631)$, $(3,727)$.  For part (b) we have $(2,7)$, $(7,337)$, and $(19, 2476081)$.  Part (b) also applies to each pair $(2, 2^p -1)$ for each Mersenne prime $2^p-1$.
As a next step, one might analyze cases which reduce to one of the above cases in a predictable number of steps.  For example

\begin{cor}\label{cor-nont1}
$\mathcal{SG}(2^k 5)=0$ if and only if $k$ is odd.
\end{cor}

\begin{proof}
Here $g(2^k 5)=2^{k-1}(6)=2^{k}3$, and so the result follows by Theorem \ref{lem-nont} (a).
\end{proof}

The authors of \cite{sum} were able to use Harold Shapiro's height function, $H(n)=H(\phi(n))+1$, to give a method for computing the nim-value of any natural number in \textsc{totient}.  Motivated by this success, they suggest analyzing a class function $\dist(n)=i$, which gives the least $i$ for which $g^i(n)$ is a prime power.  We instead analyze the function $C(n)=i$ if $g^i(n)=1$.  The initial values are:

\bigskip

\begin{tabular}{c|cccccccccccccccccccc}
    $n$ & $1$ & $2$ & $3$ & $4$ & $5$ & $6$& $7$& $8$& $9$& $10$& $11$& $12$& $13$& $14$& $15$& $16$& $17$& $18$& $19$& $20$ \\
    \hline
    $C(n)$ & $0$ & $1$ & $1$ & $2$ & $1$ & $3$ & $1$ & $3$ & $2$ & $4$ & $1$ & $4$ & $1$ & $4$ & $2$ & $4$ & $1$ & $5$ & $1$ & $5$
\end{tabular}

\newpage

\begin{lem}\label{lem-nont-C1}
    $C(4n) = C(2n) + 1$.
\end{lem}

\begin{proof}
    Note that for $k>1$ and $m$ odd, $\phi(2^k m)=\phi(2^k)\phi(m)=2^{k-1}\phi(m)=2\phi(2^{k-1}m)$.  Hence we have
    $g(4n) = 4n - \phi(4n) = 4n - 2\phi(2n) = 2(2n - \phi(2n)) = 2g(2n)$. So if $g^i(2n) = 1$, then $g^i(4n) = 2$, which means $g^{i+1}(4n) = 1$. 
\end{proof}

\begin{cor}
    $\mathcal{SG}(4n) = 1 - \mathcal{SG}(2n)$.
\end{cor}

For example, knowing that $\mathcal{SG}(10)=0$, we again obtain Corollary \ref{cor-nont1}.  We end this section with some observations about the function $C(n)$ for even $n$.

\begin{lem}\label{lem-nont-C2}
    If $n$ is even and $2^{i-1} < n \leq 2^i$, then $C(n) \geq i$.
\end{lem}

The least value we don't see equality is  $C(30)=6$.

\begin{proof}
    We proceed by induction on $i$, observing initial cases in the table above.  Suppose $2^{i-1} < n \leq 2^i$ and $n=2k$.  Then $\phi(n)\leq k$, so $g(n)\geq k>2^{i-2}$. By Lemma \ref{lem-nont-parity} $g(n)$ is also even, so we have by induction that $C(g(n))\geq i-1$.  Hence $C(n)\geq i$.    
\end{proof}

\begin{lem}
    If $p$ is an odd prime, then $C(2p) = C(2(p+1))$.
\end{lem}

\begin{proof}
 We have $g(2p)=p+1$, so $C(2p)=C(p+1)+1=C(2(p+1))$, by Lemma \ref{lem-nont-C1}.    
\end{proof}

The first time the conclusion does not hold is when $p=15$, since $C(30)=6$ and $C(32)=5$.

\begin{thm}
 Let $i\geq 1$.  The set $S_i=\{C(n)\mid 2^{i-1}\leq n\leq 2^i\text{ and }n\text{ is even}\}$ is an interval of $\mathbb{N}$.
\end{thm}

\begin{proof}
We proceed by induction, again noting initial values in the chart above.  For each $i\geq 1$, Lemma \ref{lem-nont-C2} implies that the minimal possible value of $S_i$ is $i$, and this is in fact obtained by $C(2^i)$.  

Now let $i\geq 2$, and suppose that the maximal value in $S_{i-1}$ is $M$.  By induction $S_{i-1}=\{i-1, i, \ldots, M\}$.  Lemma \ref{lem-nont-C1} then implies that $\{i, i+1, \ldots, M+1\}\subseteq S_i$.

Next suppose that $\{i, i+1, \ldots, M+1\}\neq S_i$.  Then there is some even $2^{i-1}\leq y\leq 2^i$ for which $C(y)>M+1$.  In this case, since $g(y)$ is even and $C(g(y))>M$, we must have $2^{i-1}\leq g(y)$.  Thus $C(g(y))=C(y)-1\in S_i$.
\end{proof}


\begin{thebibliography}{CCFM}

\bibitem[LIP]{LIP} M. Albert, R. Nowakowski, and D. Wolfe, 
{\em{Lessons In Play, An Introduction to Combinatorial 
Game Theory}}, A K Peters, Ltd., 2007.

\bibitem[WW1]{WW1} E. Berlekamp, J. H. Conway, and
R. Guy, {\em{Winning Ways for your Mathematical Plays}}, Academic
Press, New York, 1982. 


\bibitem[SHA]{sha}
H. Shapiro, An arithmetic function arising from the $\phi$ function, \emph{Amer. Math. Monthly} \textbf{50} (1943), 18-30 

\bibitem[IL]{sum}
D. E. Iannucci, U. Larsson, ``Game values of arithmetic functions''. \emph{Combinatorial Game Theory: A Special Collection in Honor of Elwyn Berlekamp, John H. Conway and Richard K. Guy}, edited by Richard J. Nowakowski, Bruce M. Landman, Florian Luca, Melvyn B. Nathanson, Jaroslav Nešetřil and Aaron Robertson, Berlin, Boston: De Gruyter, 2022, pp. 245-280. \url{https://doi.org/10.1515/9783110755411-014}

\end{thebibliography}
\end{document}